\documentclass[12pt]{article}

\usepackage{amsmath,amssymb,amsthm, tikz-cd}

\newtheorem{theorem}{Theorem}

\begin{document}

\title{AF-Embeddings of Graph Algebras}
\author{Christopher Schafhauser}
\date{\relax}
\maketitle

\abstract{Let $E$ be a countable directed graph.  We show that $C^*(E)$ is AF-embeddable if and only if no loop in $E$ has an entrance.  The proof is constructive and is in the same spirit as the Drinen-Tomforde desingularization in \cite{DrinenTomforde}.}

\section*{Introduction}

In \cite{PimsnerVoiculescu}, Pimnser and Voiculescu argued the irrational rotation algebras $A_\theta$ can be embedded into an AF $C^*$-algebra.  Since then, there has been an interest in characterizing the $C^*$-algebras which are AF-embeddable; especially crossed products.  Pimnser \cite{PimsnerAFE} and Brown \cite{BrownAFE}, repsectively, have solved the AF-embeddability question for algebras of the form $C(X) \rtimes \mathbb{Z}$ for a compact metric space $X$ and $A \rtimes \mathbb{Z}$ for an AF-algebra $A$.   See \cite[Chapter 8]{BrownOzawa} for a survey on AF-embeddability.

The general AF-embeddability problem is still largely unsolved.  There are only two known obstructions to AF-embeddability; namely exactness and quasidiagonality.  A $C^*$-algebra $A$ is said to be \emph{exact}, if the functor $B \mapsto A \otimes_{\text{min}} B$ preserves short exact sequences.  A $C^*$-algebra is called \emph{quasidiagonal} if there are sequences of finite dimensional $C^*$-algebras $F_n$ and completely positive contractive maps $\varphi_n: A \rightarrow F_n$ such that
\[ \left\|\varphi_n(ab) - \varphi_n(a) \varphi_n(b)\right\| \rightarrow 0 \qquad \text{and} \qquad \left\|\varphi_n(a)\right\| \rightarrow \|a\| \]
for every $a, b \in A$.  See \cite[Chapters 3 and 7]{BrownOzawa} for an introduction to exactness and quasidiagonality.

Both quasidiagonality and exactness are preserved by taking subalgebras and AF-algebras enjoy both properties.  Hence every AF-embeddable $C^*$-algebra is exact and quasidiagonal.  It is conjectured in \cite{BlackadarKirchberg} that the converse is true.  Blackadar and Kirchberg also ask if every stably finite nuclear $C^*$-algebra is quasidiagonal.  Hence in particular, the conjecture is that stable finiteness, quasidiagonality, and AF-embeddability are equivalent for nuclear $C^*$-algebras.  The main result of this paper verifies this conjecture for graph $C^*$-algebras.  In particular, we have

\begin{theorem}\label{MainTheorem}
For a countable graph $E$, the following are equivalent:
\begin{enumerate}
  \item $C^*(E)$ is AF-embeddable;
  \item $C^*(E)$ is quasidiagonal;
  \item $C^*(E)$ is stably finite;
  \item $C^*(E)$ is finite;
  \item No loop in $E$ has an entrance.
\end{enumerate}
\end{theorem}

\section*{Graph $C^*$-Algebras}

By a graph we mean a quadruple $E = (E^0, E^1, r, s)$, where $E^0$ and $E^1$ are countable sets called the $\emph{vertices}$ and $\emph{edges}$ of $E$, and $r, s: E^1 \rightarrow E^0$ are functions called the \emph{range} and \emph{source} maps.  Given a graph $E$, a Cuntz-Krieger $E$-family in a $C^*$-algebra $A$ is a collection
\[ \{p_v, s_e : v \in E^0, e \in E^1\} \subseteq A \]
such that for all $v \in E^0$ and $e, f \in E^1$, we have
\begin{enumerate}
  \item $p_v^2 = p_v = p_v^*$ \quad for all $v \in E^0$
  \item $s_e^* s_f = \begin{cases} \, p_{s(e)} & e = f \\ \, 0 & e \neq f \end{cases}$
  \item $\displaystyle p_v = \sum_{e \in r^{-1}(v)} s_e s_e^*$ \quad if $0 < |r^{-1}(v)| < \infty$.
\end{enumerate}
Let $C^*(E)$ denote the universal $C^*$-algebra generated by a Cuntz-Krieger $E$-family.  See \cite{Raeburn} for an introduction to graph $C^*$-algebras.

If $E$ is a graph and $n \geq 1$, a path in $E$ is a list of edges $\alpha = (\alpha_n, \ldots, \alpha_1)$ such that $r(\alpha_i) = s(\alpha_{i+1})$ for each $1 \leq i < n$.  Define $r(\alpha) = r(\alpha_n)$ and $s(\alpha) = s(\alpha_1)$.  Define $E^n$ to be the set of paths of length $n$ in $E$ and $E^* = \bigcup_{n=0}^\infty E^n$ the paths of finite length in $E$.  In particular, the vertices of $E$ are considered to be paths of length 0.  Given $\alpha = (\alpha_n, \ldots, \alpha_1)$, define $s_\alpha = s_{\alpha_n} \cdots s_{\alpha_1}$.  It can be shown that
\[ C^*(E) = \overline{\operatorname{span}} \{ s_\alpha s_\beta^* : \alpha, \beta \in E^* \text{ with } s(\alpha) = s(\beta) \}. \]

A loop in $E$ is a path $\alpha \in E^n$ with $n \geq 1$ such that $r(\alpha) = s(\alpha)$.  We say $\alpha$ is a \emph{simple loop} if $r(\alpha_i) \neq r(\alpha_j)$ for $i \neq j$.  We say $\alpha$ has an entrance if $|r^{-1}(r(\alpha_i))| > 1$ for some $i$.  The structure of the algebra $C^*(E)$ is closely related to the structure of the loops in $E$.  We will show in Theorem \ref{MainTheorem}, the AF-embeddability of $C^*(E)$ is also characterized by the loops in $E$.

We recall two results about graph $C^*$-algebras.  Theorem \ref{AFGraphAlgebra} is from Kumijan, Pask, and Raeburn in the row-finite case and Drinen and Tomforde in general (see \cite[Theorem 2.4]{KPR} and \cite[Corollary 2.13]{DrinenTomforde}). Theorem \ref{CKUnqiueness} is Szyma\'{n}ski's generalization of the Cuntz-Krieger Uniqueness Theorem (see \cite[Theorem 1.2]{Szymanski}).

\begin{theorem}\label{AFGraphAlgebra}
For a countable graph $E$, $C^*(E)$ is AF if and only if $E$ has no loops.
\end{theorem}

\begin{theorem}\label{CKUnqiueness}
Suppose $E$ is a graph, $A$ is a $C^*$-algebra, and $\{\tilde{p}_v, \tilde{s}_e\} \subseteq A$ is a Cuntz-Kreiger $E$-family.  If $\tilde{p}_v \neq 0$ for every $v \in E^0$ and $\sigma(\tilde{s}_\alpha) \supseteq \mathbb{T}$ for every entry-less loop $\alpha \in E^*$, then the induced morphism $C^*(E) \rightarrow A$ defined by $p_v \mapsto \tilde{p}_v$ and $s_e \mapsto \tilde{s}_e$ is injective.
\end{theorem}

\section*{Proof of Theorem \ref{MainTheorem}}

We are now ready to prove our main result.  Starting with a graph $E$ satisfying condition (5), we will replace each loop in $E$ with the Bratteli diagram of an AF-algebra to build a new graph $F$ such that $C^*(F)$ is AF and $C^*(E) \subseteq C^*(F)$.  The idea of the proof is motivated by the Drinen-Tomforde desingularization process introduced in \cite{DrinenTomforde}.

\begin{proof}[Proof of Theorem \ref{MainTheorem}]
It is well-known that (1) implies (2) and (2) implies (3) (see \cite[Propositions 7.1.9, 7.1.10, and 7.1.15]{BrownOzawa}) and it is obvious that (3) implies (4).  To see (4) implies (5), note that if $\alpha, \beta \in E^*$ are distinct paths with $s(\alpha) = r(\alpha) = r(\beta)$, then we have
\[ s_\alpha^* s_\alpha = p_{s(\alpha)} \quad \text{and} \quad s_\alpha s_\alpha^* \lneqq s_\alpha s_\alpha^* + s_\beta s_\beta^* \leq p_{s(\alpha)}. \]
So $p_{s(\alpha)}$ is an infinite projection and $C^*(E)$ is infinite.

Now suppose (5) holds.  Choose a unital AF-algebra $A$ such that there is a unitary $t \in A$ with $\sigma(t) = \mathbb{T}$ and let $B$ be a Bratteli diagram for $A$ with sink $v$.  Let $e_n \cdots e_2 e_1$ be a simple loop in $E$ and set $u_i = s(e_i)$.  Define a graph $F$ by
\[ F^0 = E^0 \cup B^0, \qquad F^1 = \left(E^1 \setminus \{e_1, \ldots, e_n \} \right) \cup B^1 \cup \{f_1, \ldots f_n\} \]
and extend the range and source maps by $r(f_i) = u_i$ and $s(f_i) = v$.  For example, if $A = M_{2^\infty}$, and $E$ and $B$ the graphs

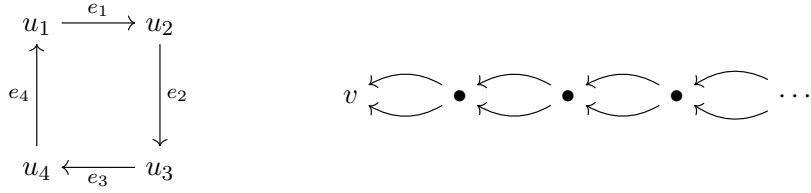
\begin{center}
\begin{tikzcd}[row sep = small]
&&&&&&& \\
u_1 \arrow{r}{e_1} & u_2 \arrow{dd}{e_2} &&&&&& \\
& & & v & \bullet \arrow[bend left]{l}\arrow[bend right]{l} & \bullet \arrow[bend left]{l}\arrow[bend right]{l} & \bullet \arrow[bend left]{l}\arrow[bend right]{l} & \cdots \arrow[bend left]{l}\arrow[bend right]{l} \\
u_4 \arrow{uu}{e_4} & u_3 \arrow{l}{e_3} &&&&&& \\
&&&&&&&
\end{tikzcd}
\end{center}

\noindent then $F$ is the graph given below:

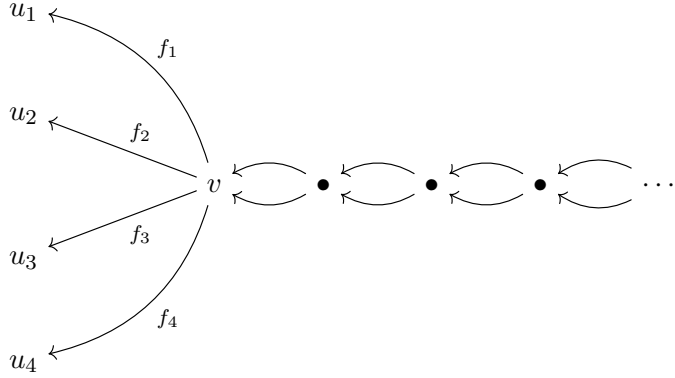
\begin{center}
\begin{tikzcd}[row sep = small]
&&&&&&&\\
 u_1  & & & &  &  &  & \\ \\
 u_2  & & & &  &  &  & \\
 & &  v \arrow[bend right]{lluuu}[swap]{f_1} \arrow{llu}[swap]{f_2} \arrow{lld}{f_3} \arrow[bend left]{llddd}{f_4} & \bullet \arrow[bend left]{l}\arrow[bend right]{l} & \bullet \arrow[bend left]{l}\arrow[bend right]{l} & \bullet \arrow[bend left]{l}\arrow[bend right]{l} & \cdots \arrow[bend left]{l}\arrow[bend right]{l} \\
 u_3  & & & &  &  &  & \\ \\
 u_4  & & & &  &  &  & \\
&&&&&&&
\end{tikzcd}
\end{center}

Note that $p_vC^*(F)p_v \cong A$ and hence we may view $t$ as an element of $C^*(F)$.  Define $\tilde{s}_{e_i} = s_{f_{i+1}} t s_{f_i}^* \in C^*(F)$ for each $i = 1, \ldots, n$.  Since no loop in $E$ has an entrance, we have $r_F^{-1}(u_i) = \{f_i\}$.  Hence
\[ \tilde{s}_{e_i}^* \tilde{s}_{e_i} = s_{f_i} s_{f_i}^* = p_{u_i} \quad \text{and} \quad \tilde{s}_{e_i} \tilde{s}_{e_i}^* = s_{f_{i+1}} s_{f_{i+1}}^* = p_{u_{i+1}}. \]
Moreover,
\[ \sigma(\tilde{s}_{e_n} \tilde{s}_{e_{n-1}} \cdots \tilde{s}_{e_1}) = \sigma(s_{f_1} t^n s_{f_1}^*) = \sigma(s_{f_1}^*s_{f_1} t^n) = \sigma(t^n) = \mathbb{T} \cup \{0\}. \]
Now, by Theorem \ref{CKUnqiueness}, there is an inclusion $C^*(E) \hookrightarrow C^*(F)$ given by
\[ p_v \mapsto p_v \text{ for } v \in E^0 \quad \text{and} \quad s_e \mapsto \begin{cases} \tilde{s}_e & e \in \{e_1, \ldots, e_n\}, \\ s_e & e \in E^1 \setminus \{e_1, \ldots, e_n\}. \end{cases} \]

Note that since no loop in $E$ has an entrance, the loops in the graph $E$ are disjoint.  Thus by applying the construction above to every loop in $E$, we may build a graph $F$ with no loops and an embedding $C^*(E) \hookrightarrow C^*(F)$.  Since $F$ has no loops, $C^*(F)$ is AF by Theorem \ref{AFGraphAlgebra} and hence $C^*(E)$ is AF-embeddable.
\end{proof}

\noindent Dept.\ of Mathematics, University of Nebraska-Lincoln, Lincoln, NE, 68588-0130 \\
\emph{E-mail address:} cschafhauser2@math.unl.edu


\begin{thebibliography}{\relax}
\bibitem{BlackadarKirchberg} B. Blackadar and E. Kirchberg, \emph{Generalized inductive limits of ﬁnitedimensional $C^*$-algebras}, Math. Ann. 307 (1997), 343 - 380.
\bibitem{BrownAFE} N. P. Brown, \emph{AF Embeddability of Crossed Products of AF Algebras by the Integers}, J. Funct. Anal. \textbf{160} (1998), 150-175.
\bibitem{BrownOzawa} N. P. Brown and N. Ozawa, \emph{$C^*$-Algebras and Finite-Dimensional Approximations}. Graduate Studies in Mathematics, 88. Amer. Math. Soc., Providence, RI, 2008.
\bibitem{DrinenTomforde} D. Drinen and M. Tomforde, \emph{The $C^*$-algebras of arbitrary graphs}, Rocky Mountain J. Math. 35 (2005), no. 1, 105–135.
\bibitem{KPR} A. Kumijan, D. Pask, and I. Raeburn, \emph{Cuntz-Krieger algebras of directed graphs}, Pacific J. Math. \textbf{184} (1998), 161-174.
\bibitem{PimsnerAFE} \emph{Embedding some transformation group $C^*$-algebras into AF-algebras}, Ergodic Theory Dynam. Systems 3 (1983), no. 4, 613–626.
\bibitem{PimsnerVoiculescu} M. Pimsner and D. Voiculescu,  \emph{Imbedding the irrational rotation $C^*$-algebra into an AF-algebra}, J. Operator Theory {\bf 4} (1980), no. 2, 201--210.
\bibitem{Raeburn} I. Raeburn, {\em Graph algebras},  CBMS Regional Conference Series in Mathematics {\bf 103} Published for the Conference Board of the Mathematical Sciences, Washington D.C.~by the AMS, Providence, RI (2005).
\bibitem{Szymanski} W. Szyma\'{n}ski, {\em General Cuntz-Krieger uniqueness theorem}, International J.~Math. 13 (2002).
\end{thebibliography}
\end{document}